\documentclass[a4paper,12pt,reqno]{amsart}
\usepackage[a4paper,margin=1.9cm,top=2.5cm,bottom=2.5cm,centering,vcentering]{geometry}
\usepackage{amsmath,amssymb,graphicx,amsfonts,tikz}
\usetikzlibrary{arrows,matrix}
\usepackage{amsfonts,hyperref} 
\usepackage{amsthm}
\newtheorem{theorem}{Theorem}[section]

\newtheorem{remark}{Remark}[section]
\usepackage[utf8]{inputenc}
\usepackage[english]{babel}
\newcommand{\vin}{\vdash}

\title{Biases in Non-Unitary Partitions}

       \author[P. J. Mahanta]{Pankaj Jyoti Mahanta}
    \address[P. J. Mahanta]{Department  of Mathematical Sciences, Tezpur University, Napaam,  Tezpur 784028, Assam, India}
    \email{pjm2099@gmail.com}
    
\author[M. P. Saikia]{Manjil P. Saikia}
\address[M. P. Saikia]{Mathematical and Physical Sciences division, School of Arts and Sciences, Ahmedabad University, Ahmedabad 380009, Gujarat, India}
\email{manjil@saikia.in}

\author[A. Sarma]{Abhishek Sarma}
\address[A. Sarma]{Department  of Basic Science, Assam Skill University, Mangaldai, 784125, Assam, India}
\email{abhitezu002@gmail.com}

\date{\today}


\begin{document}

\begin{abstract}
    Recently, the concept of parity bias in integer partitions has been studied by several authors. We continue this study here, but for non-unitary partitions (namely, partitions with parts greater than $1$). We prove analogous results for these restricted partitions to those that have been obtained by Kim, Kim, and Lovejoy (2020) and Kim and Kim (2021). We also look at inequalities between two classes of partitions studied by Andrews (2019), where the parts are separated by parity (either all odd parts are smaller than all even parts or vice versa).
\end{abstract}
\maketitle

\section{Introduction}
A partition $\lambda=(\lambda_1, \lambda_2, \ldots, \lambda_k)$ of $n$ is a non-increasing sequence of natural numbers, $\lambda_1\geq \lambda_2\geq \cdots \geq \lambda_k$ such that $\lambda_1+\lambda_2+\cdots+\lambda_k=n$. Here, each $\lambda_i$ is called a part of the partition $\lambda$ of $n$ (written as $\lambda\vin n$) and the length of the partition, denoted by $\ell(\lambda)$ is $k$. Partitions have been studied since the time of Euler, and continue to be a serious topic for ongoing research in several directions. A good introduction to the subject is given in the masterly treatment of Andrews \cite{gea1}.

In the theory of partitions, inequalities arising between two classes of partitions have a long tradition of study, see for instance work in this direction by Alder \cite{alder},  Andrews \cite{gea3}, McLaughlin \cite{Mc}, Chern, Fu, and Tang \cite{Chern} and Berkovich and Uncu \cite{Berkovich}, among others. In 2020, Kim, Kim and Lovejoy \cite{kim2020parity} introduced a phenomenon in integer partitions called \textit{parity bias}, wherein the number of partitions of $n$ with more odd parts (denoted by $p_o(n)$) are more in number than the number of partitions of $n$ with more even parts (denoted by $p_e(n)$). That is, they proved for $n\neq 2$, $p_o(n)>p_e(n)$. They also conjectured a similar inequality for partitions with only distinct parts. For $n>19$, they conjectured that $d_o(n)>d_e(n)$, where $d_o(n)$ (resp. $d_e(n)$) denotes the number of partitions of $n$ with distinct parts with more odd parts (resp. even parts) than even parts (resp. odd parts). Further generalizations of the results of Kim, Kim, and Lovejoy \cite{kim2020parity} have been found by Kim and Kim \cite{zbMATH07394382} and Chern \cite{zbMATH07481756}. Most of the proofs of the results in these papers use techniques arising from $q$-series methods.

The first two authors, in collaboration with Banerjee, Bhattacharjee, and Dastidar \cite{zbMATH07524341}, proved both the above-quoted result and conjecture of Kim, Kim, and Lovejoy \cite{zbMATH07394382} using combinatorial means. In addition, they proved several more results on parity biases of partitions with restrictions on the set of parts. For a nonempty set $S \subsetneq \mathbb{Z}_{\geq 0}$, define
\[P_e^S(n):=\{\lambda \in P_e(n): \lambda_i \notin S\}\]
and
\[P_o^S(n):=\{\lambda \in P_o(n): \lambda_i \notin S\},\]
where the set $P_e(n)$ (resp. $P_o(n)$) consists of all partitions of $n$ with more even parts (resp. odd parts) than odd parts (resp. even parts). Let us denote the number of partitions of $P_e^S(n)$ (resp. $P_o^S(n)$) by $p_e^S(n)$ (resp. $p_o^S(n)$).
Banerjee \emph{et al.} \cite{zbMATH07524341} proved the following result.
\begin{theorem}[Banerjee \emph{et al.}, \cite{zbMATH07524341}]\label{thm:reverse}
For positive integers $n$, the following inequalities are true (the range is given in the brackets),
	\begin{equation}\label{eq:ineq}
	p_o^{\{1\}}(n)<p_e^{\{1\}}(n), \quad (n>7),
	\end{equation}
\begin{equation}\label{eq:22}
    p_o^{\{2\}}(n)>p_e^{\{2\}}(n), \quad (n\geq 1),
\end{equation}
and
\begin{equation}\label{eq:33}
    p^{\{1,2\}}_o(n)>p^{\{1,2\}}_e(n), \quad (n>8).
\end{equation}
\end{theorem}
\noindent All of the proofs of the above inequalities were by using combinatorial techniques. Although they do not use this term, partitions where the part $1$ does not appear are called \textit{non-unitary partitions}, and we will use this terminology in this paper.

In 2023, B. Kim and E. Kim \cite{kimdiscrete} gave two further refinements for parity biases in ordinary integer partitions. For the first refinement, they let $p(m,n)$ be the number of partitions of $n$ with the number of odd parts minus the number of even parts to be $m$. They proved the following result.
\begin{theorem}\cite[Theorem 1]{kimdiscrete}
    For a positive integer $m\geq0$, we have
    \begin{align*}
        p(m,n) \geq p(-m,n).
    \end{align*}
\end{theorem}
\noindent The second refinement is that parity bias still holds if any odd part $\geq3$ is not allowed. This is given by the following theorem.
\begin{theorem}\cite[Theorem 2]{kimdiscrete}
    Let $k$ be a positive integer. Then, for all positive integers $n$,
    \begin{align*}
        p^{\{2k+1\}}_o(n)>p^{\{2k+1\}}_e(n).
    \end{align*}
\end{theorem}
\noindent The proofs of these results involve both combinatorial and analytic techniques. In 2024, they  \cite{kimramanujan} looked at some asymptotic results related to parity biases, which we do not mention here.

In view of \eqref{eq:ineq}, it is clear that the contribution of $1$ towards parity bias is much more than that of other odd parts.

The primary goal of this paper is to use analytical techniques and prove results of the type proved by Banerjee \emph{et al.}, that is, about parity biases in partitions with certain restrictions on their allowed parts. We prove the inequality \eqref{eq:ineq} using analytical techniques, as well as prove results in a similar setup for the biases discussed in the work of Kim and Kim \cite{zbMATH07394382}. We further look at some simply derived results on biases in partitions with a restriction on the size of the minimum part as well as on parity of the number of parts of a given parity. Our techniques can also be used to prove partition inequalities of the type where the number of partitions of a certain class of partitions is more than another class. This is explored for two classes of partitions studied by Andrews \cite{andrews}, where the parts are separated by parity, where either all odd parts are smaller than all even parts or vice versa.

The paper is structured as follows: in Section \ref{sec:prelim} we collect some $q$-series identities which we will use later, in Section \ref{sec:results} we state and prove our main results, namely on biases in ordinary non-unitary partitions, in Section \ref{sec:pod} we look at inequalities on partitions with parts separated by parity. Finally, we close the paper with some concluding remarks in Section \ref{sec:conc}.

\section{Preliminaries}\label{sec:prelim}

We need some preliminaries before going into our results. We use the standard $q$-series notation
\[
(a)_{n} = (a;q)_{n} = \prod_{k=1}^{n}(1-aq^{k-1}), \quad |q|<1,
\]
and,
\[
(a_1,a_2,\ldots,a_m;q)_n:=(a_1;q)_n(a_2;q)_n\cdots (a_m;q)_n.
\]
Also, recall Heine's transformation \cite[Appendix III.1]{gasper2004basic}, which says for $|z|, |q|, |b| \leq$ 1, we have
\begin{equation}\label{heine}
    \sum_{n\geq0}\frac{(a)_{n}(b)_{n}}{(q)_{n}(c)_{n}}z^n = \frac{(b)_{\infty}(az)_{\infty}}{(c)_{\infty}(z)_\infty}\sum_{n\geq0}\frac{(z)_{n}(c/b)_{n}}{(q)_{n}(az)_{n}}b^n.
\end{equation}
By appropriately iterating Heine's transformation,  we obtain \cite[Appendix III.3]{gasper2004basic} what is sometimes called the $q$-analogue of Euler's transformation, which says that for $|z|$, $|{\frac{abz}{c}}|$ $\leq$ 1, we have
\begin{align}\label{euler}
\sum_{n\geq0}\frac{(a)_{n}(b)_{n}}{(q)_{n}(c)_{n}}z^n = \frac{(abz/c)_{\infty}}{(z)_\infty}\sum_{n\geq0}\frac{(c/a)_{n}(c/b)_{n}}{(q)_{n}(c)_{n}}(abz/c)^n.
\end{align}

The final set of auxiliary identities that we need is described below. Due to Euler \cite[p. 19]{gea1}, we know that
\begin{equation*}
  \frac{1}{(a;q)_{\infty}}=\sum_{n\geq0} \frac{a^{n}}{(q;q)_{n}}  
\end{equation*}
Therefore,
\begin{equation}\label{eq:euler-2}
   \frac{1}{(q;q^2)_{\infty}}=\sum_{n\geq0} \frac{q^{n}}{(q^2;q^2)_{n}}=\sum_{n\geq0} \frac{q^{n}}{(-q)_{n}(q)_{n}}, 
\end{equation}
and
\begin{equation}\label{eq:euler-3}
    \frac{1}{(q^2;q^2)_{\infty}}=\sum_{n\geq0} \frac{q^{2n}}{(q^2;q^2)_{n}}=\sum_{n\geq0} \frac{q^{2n}}{(-q)_{n}(q)_{n}}.
\end{equation}

Now, Substituting $c=-q, a, b \to 0, z = q$ in equation \eqref{heine} we get
\begin{equation}\label{eq:heine-1}
    \sum_{n\geq0} \frac{q^{n}}{(-q)_{n}(q)_{n}}=\frac{1}{(-q)_\infty(q)_\infty} \sum_{n\geq0} q^{\frac{n^2+n}{2}}.
\end{equation}
Again, substituting $c=-q, a, b \to 0, z = q^2$ in equation \eqref{heine} we get
\begin{equation}\label{eq:heine-2}
    \sum_{n\geq0} \frac{q^{2n}}{(-q)_{n}(q)_{n}}=\frac{1}{(-q)_\infty(q)_\infty} \sum_{n\geq0} (1-q^{n+1}) q^{\frac{n^2+n}{2}}.
\end{equation}

Finally, from \cite[p. 21, Eq. (2.2.9)]{gea1}, we also have
\begin{align}\label{eq:And98 p.21}
    \frac{1}{(q^2;q^2)_{\infty}}=\sum_{n\geq0}\frac{q^{2n^2}}{(q^2;q^2)^2_{n}}.
\end{align}

\section{Biases in Ordinary Non-Unitary Partitions}\label{sec:results}

Using analytical techniques, we will give a proof of the following result, which was proved by Banerjee \textit{et al.} \cite{zbMATH07524341} combinatorially. We modify the notation a bit and let $q_e(n)$ (resp. $q_o(n)$) be the number of non-unitary partitions of $n$ where the number of even (resp. odd) parts is more than the number of odd (resp. even) parts.

\begin{theorem}[Theorem 1.5, \cite{zbMATH07524341}]\label{thm:mm}
For all positive integers $n\geq8$, we have
\[
	q_o(n)<q_e(n).
	\]
\end{theorem}

Let $p_{j,k,m}(n)$ be the number of partitions of $n$ such that there are more parts congruent to $j$ modulo $m$ than parts congruent to $k$ modulo $m$, for $m\geq 2$. Then, Kim and Kim \cite{zbMATH07394382} proved that for all positive integers $n\geq m^2-m+1$, we have
\[
p_{1,0,m}(n)>p_{0,1,m}(n).
\]
Let us now denote by $q_{j,k,m}(n)$ the number of non-unitary partitions of $n$ such that there are more parts congruent to $j$ modulo $m$ than parts congruent to $k$ modulo $m$, for $m > 2$. Then, we have the following result.

\begin{theorem}\label{thm:kim-new}
For $n \geq 4m+3$ and $m\geq 2$, we have
\[
q_{0,1,m} (n) > q_{1,0,m} (n).
\]
\end{theorem}

By standard combinatorial arguments, we have that $\dfrac{q^{bn}}{(q^2;q^2)_{n}}$ is the generating function for partitions with exactly $n$ odd parts with the minimum odd part being at least $b$, as well as it is the generating function for partitions with exactly $n$ even parts with the minimum even part being at least $b$. We will use this in the proofs below without commentary.

\begin{proof}[Proof of Theorem \ref{thm:mm}]

Let $P_o(q)$ (resp. $P_e(q)$) be the generating functions of $q_o(n)$ (resp. $q_e(n)$). The generating functions are not difficult to see. 

Observe that $\sum_{n\geq0} \dfrac{q^{3n}}{(q^2;q^2)^2_{n}}=\sum_{n\geq0} \dfrac{1}{(q^2;q^2)_{n}}\cdot\dfrac{q^{3n}}{(q^2;q^2)_{n}}$ is the generating function of partitions where the number of even parts is less than or equal to the number of odd parts where the minimum odd part is at least 3. Similarly, $\sum_{n\geq0} \dfrac{q^{5n}}{(q^2;q^2)^2_{n}}=\sum_{n\geq0} \dfrac{q^{2n}}{(q^2;q^2)_{n}}\cdot\dfrac{q^{3n}}{(q^2;q^2)_{n}}$ is the generating function of partitions where the number of even parts is equal to the number of odd parts where the minimum odd part is at least 3.
Thus, we have
\begin{align*}
P_{o}(q) & = \sum_{n\geq0} \frac{q^{3n}}{(q^2;q^2)^2_{n}} - \sum_{n\geq0} \frac{q^{5n}}{(q^2;q^2)^2_{n}} = q^3+q^5+q^6+q^7+2q^8+\cdots.
\end{align*}
Again, $\dfrac{1}{(q^2;q)_{\infty}}$ is the generating function of non-unitary partitions. If we subtract $\sum_{n\geq0} \dfrac{q^{3n}}{(q^2;q^2)^2_{n}} $ from $\dfrac{1}{(q^2;q)_{\infty}}$, we will be get the generating function of partitions with more even parts than odd parts where the minimum odd part is at least 3. Hence, we have
\begin{align*}
P_{e}(q) & = \dfrac{1}{(q^2;q)_{\infty}} - \sum_{n\geq0} \frac{q^{3n}}{(q^2;q^2)^2_{n}}= q^2+2q^4+3q^6+q^7+5q^8+\cdots.
\end{align*}

Substituting $c=q^4, a, b \to 0, z = q^3, q\to q^2$ in equation \eqref{euler} we get
\begin{align*}
P_{o}(q) 
& = \sum_{n\geq1} \frac{q^{3n}}{(q^2;q^2)^2_{n}}(1-q^{2n})\\ 
& = \frac{1}{(1-q^2)}\sum_{n\geq1} \frac{q^{3n}}{(q^4;q^2)_{n-1}(q^2;q^2)_{n-1}} = \frac{q^3}{(1-q^2)}\sum_{n\geq0} \frac{q^{3n}}{(q^4;q^2)_{n}(q^2;q^2)_{n}}\\
& = \frac{1}{(q^3;q^2)_{\infty}}\sum_{n\geq0} \frac{q^{2n^2+5n+3}}{(q^2;q^2)_{n+1}(q^2;q^2)_{n}} = \frac{1}{(q^3;q^2)_{\infty}}\sum_{n\geq1} \frac{q^{2n^2+n}}{(q^2;q^2)_{n}(q^2;q^2)_{n-1}}\\
& = \frac{1}{(q^3;q^2)_{\infty}}\sum_{n\geq1}\frac{q^{2n^2+n}}{(q^2;q^2)^2_{n}}(1-q^{2n}).
\end{align*}

Substituting $c=q^2, a, b \to 0, z = q^3, q\to q^2$ in equation \eqref{euler}, we get
\begin{align*}
P_{e}(q) & = \frac{1}{(q^3;q^2)_{\infty}}\frac{1}{(q^2;q^2)_{\infty}} - \sum_{n\geq0} \frac{q^{3n}}{(q^2;q^2)^2_{n}}\\
& = \frac{1}{(q^3;q^2)_{\infty}}\sum_{n\geq0}\frac{q^{2n^2}}{(q^2;q^2)^2_{n}} - \frac{1}{(q^3;q^2)_{\infty}}\sum_{n\geq0}\frac{q^{2n^2+3n}}{(q^2;q^2)^2_{n}} \quad(\text{thanks to \eqref{eq:And98 p.21}})\\
& = \frac{1}{(q^3;q^2)_{\infty}}\sum_{n\geq1}\frac{q^{2n^2}}{(q^2;q^2)^2_{n}}(1-q^{3n}).
\end{align*}
Now,
\begin{align}
P_{e}(q) - P_{o}(q) = \frac{1}{(q^3;q^2)_{\infty}}\sum_{n\geq1}\frac{q^{2n^2}}{(q^2;q^2)^2_{n}}(1-q^{n}).\label{p_e-p_o}
\end{align}

Clearly, for the summand from $n=2$ onward the coefficients are positive, because if $n$ is even, then $1-q^n$ will be canceled by a factor of $(q^2;q^2)_{n}$ and if $n$ is odd, then it will be canceled by a factor of $(q^3;q^2)_{\infty}$. 

From \cite[Eq. (3.4)]{kim2020parity}, we recall that
\begin{align*}
	\frac{1}{(q^3;q^2)_{\infty}}\frac{q(1-q)}{(1-q^2)^2} = -q^2-q^4+\frac{q(1+q^2)}{(1-q^2)}+\frac{q}{1-q^2}\sum_{n=2}^{\infty}\frac{(-q^2)_{n-1}}{(q^2)_{n-1}}(1+q^{2n+1})q^{\frac{3n^2+n}{2}}.
\end{align*}
Multiplying both sides of the above by $q$, we have
\begin{align}
\frac{1}{(q^3;q^2)_{\infty}}\frac{q^2(1-q)}{(1-q^2)^2} = -q^3-q^5+\frac{q^2(1+q^2)}{(1-q^2)}+\frac{q^2}{1-q^2}\sum_{n=2}^{\infty}\frac{(-q^2)_{n-1}}{(q^2)_{n-1}}(1+q^{2n+1})q^{\frac{3n^2+n}{2}},\label{posn1}
\end{align}
where the left side of \eqref{posn1} is the case $n=1$ in \eqref{p_e-p_o}.

We see that the coefficients for all terms are nonnegative except for $q^3$ and $q^5$. The terms of the expansion of the third summand of the RHS consist of terms of the form $q^{2i}$ for all $i\in \mathbb{N}$. For $n=2$, the fourth summand of the RHS gives a series where the terms are of the form  $q^{2i+1}$ for all $i\in \mathbb{N}$ and $i\geq4$. For all $n>2$, the minimum power of $q$ in the expansion of the fourth term of the RHS is greater than 9. Also, for all $n>1$, the minimum power of $q$ in the expansion of $P_{e}(q) - P_{o}(q)$ is greater than or equal to 8. So, in each case, the coefficient of $q^7$ is 0. This completes the proof.
\end{proof}

\begin{proof}[Proof of Theorem \ref{thm:kim-new}]
We start by acknowledging the fact that $\dfrac{q^{bn}}{(q^m;q^m)_{n}}$ is the generating function with partitions into $n$ parts congruent to $b \pmod m$. Let $P_{1,0,m}(q)$ (resp. $P_{ 0,1 m}(q)$) be the generating functions of $q_{1, 0, m}(n)$ (resp. $q_{ 0,1, m}(n)$). The reasoning for the generating functions are along similar lines as that of the above proof. However, we need to multiply $\dfrac{(q^{m+1},q^m;q^m)_{\infty}}{(q^2;q)_{\infty}}$ with $\sum_{n\geq0} \dfrac{q^{(m+1)n}}{(q^m;q^m)^2_{n}} -\sum_{n\geq0} \dfrac{q^{(m+1)n+mn}}{(q^m;q^m)^2_{n}}$, since this part represents the unrestricted parts as the bias is only on the parts of the form $0\pmod{m}$ and $1\pmod{m}$. Thus, we have
\begin{align*}
      P_{1,0,m}(q) = \frac{(q^{m+1},q^m;q^m)_{\infty}}{(q^2;q)_{\infty}} \sum_{n\geq0} \frac{q^{(m+1)n}}{(q^m;q^m)^2_{n}} -    \frac{(q^{m+1},q^m;q^m)_{\infty}}{(q^2;q)_{\infty}} \sum_{n\geq0} \frac{q^{(m+1)n+mn}}{(q^m;q^m)^2_{n}},
\end{align*}
and
\begin{align*}
    P_{0,1, m}(q) =\frac{1}{(q^2;q)_{\infty}} - \frac{(q^{m+1},q^m;q^m)_{\infty}}{(q^2;q)_{\infty}} \sum_{n\geq0} \frac{q^{(m+1)n}}{(q^m;q^m)^2_{n}}.
\end{align*}
Now, 
\begin{align}\label{eq:new-11}
P_{1,0,m}(q) 
\nonumber & = \frac{(q^{m+1},q^m;q^m)_{\infty}}{(q^2;q)_{\infty}} \sum_{n\geq0}\frac{q^{(m+1)n}}{(q^m;q^m)^2_{n}}(1-q^{mn})\\
\nonumber & = \frac{(q^{m+1},q^m;q^m)_{\infty}}{(q^2;q)_{\infty}} \sum_{n\geq1}\frac{q^{(m+1)n}}{(q^m;q^m)_{n}(q^m;q^m)_{n-1}}\\
\nonumber & = \frac{(q^{m+1},q^m;q^m)_{\infty}}{(q^2;q)_{\infty}}\frac{q^{m+1}}{(1-q^m)} \sum_{n\geq0}\frac{q^{(m+1)n}}{(q^m,q^{2m};q^m)_{n}}\\
\intertext{(by substituting, $q \to q^m, a,b \to 0, c \to q^{2m}$ and $z \to q^{m+1}$ in equation \eqref{euler}, we get)}
\nonumber& = \frac{(q^m;q^m)_{\infty}}{(q^2;q)_{\infty}}\frac{q^{m+1}}{(1-q^m)} \sum_{n\geq0}\frac{q^{mn^2+2mn+n}}{(q^m,q^{2m};q^m)_{n}}\\
& = \frac{(q^m;q^m)_{\infty}}{(q^2;q)_{\infty}} \sum_{n\geq1}\frac{q^{mn^2+n}(1-q^{mn})}{(q^m;q^m)^2_{n}}.
\end{align}

Similarly, we have
\begin{align}\label{eq:new-22}
P_{0, 1, m}(q) & = \nonumber \frac{(q^m;q^m)_{\infty}}{(q^2;q)_{\infty}}\sum_{n\geq0}\frac{q^{mn^2}}{(q^m;q^m)^2_{n}} - \frac{(q^m;q^m)_{\infty}}{(q^2;q)_{\infty}}\sum_{n\geq0} \frac{q^{mn^2+(m+1)n}}{(q^m;q^m)^2_{n}}\\ 
& = \frac{(q^m;q^m)_{\infty}}{(q^2;q)_{\infty}}\sum_{n\geq1}\frac{q^{mn^2}}{(q^m;q^m)^2_{n}}(1 - q^{(m+1)n}).
\end{align}
From equations \eqref{eq:new-11} and \eqref{eq:new-22}, we get
\begin{align*}
P_{0, 1, m}(q) - P_{1,0,m}(q) = \frac{(q^m;q^m)_{\infty}}{(q^2;q)_{\infty}} \sum_{n\geq1}\frac{q^{mn^2}}{(q^m;q^m)^2_{n}}(1-q^n).
\end{align*}
From Kim and Kim \cite[Lemma 2.1]{zbMATH07394382}, we see that the above difference has nonnegative coefficients for all $q^k$ with $k>2m+1$. The summand $n=2$ is $\dfrac{(q^m;q^m)_{\infty}q^{4m}}{(q^3;q)_{\infty}(q^m;q^m)^2_{2}}$. This shows that coefficients of $q^k$ are positive for $k \geq 4m+3$. In fact, the coefficient of $q^{4m}$ is also positive. So, we have our result.
\end{proof}

\section{Inequalities between Partitions with Parts Separated by Parity}\label{sec:pod}

Andrews \cite{andrews-1, andrews} studied partitions in which parts of a given parity are all smaller than those of the other parity, and proved several interesting results, which have been studied by other authors as well. We denote by $p_{yz}^{wx}(n)$ the cardinalities of the class of partitions of $n$ studied by Andrews. The symbols $wx$ and $yz$ are formed with the first letter either $e$ or $o$ (denoting even or odd parts) and the second letter either $u$ or $d$ (denoting unrestricted or distinct parts). The parts separated by the symbol in the subscript are assumed to lie below the parts represented by the superscript. This gives rise to eight different families of partitions, namely $p_{eu}^{ou}(n), p_{eu}^{od}(n), p^{eu}_{ou}(n), p^{ed}_{ou}(n), p^{ou}_{ed}(n), p^{od}_{ed}(n), p_{od}^{eu}(n)$ and $p_{od}^{ed}(n)$. The corresponding generating functions for the class of partitions counted by $p_{yx}^{wz}(n)$ are denoted by \[P_{yx}^{wz}(q):=\sum_{n\geq 0}p_{yx}^{wz}(n)q^n.\] The corresponding set of all partitions counted by $p_{yx}^{wz}(n)$ is denoted by $P_{yx}^{wz}(n)$. Collectively, we call all such partitions as partitions with parts separated by parity.

Recently, Ballantine and Welch \cite{ballantine2024combinatorial} proved the following inequalities for such partitions with parts separated by parity with some additional conditions:
\begin{align*}
    p_{ed}^{od}(n)&<p_{od}^{ed}(n), \text{for $n\geq11$,}\\
    p_{eu}^{ou}(n)&<p_{ou}^{eu}(n), \text{for $n\geq3$,}\\
    p_{od}^{eu}(n)&<p_{ed}^{ou}(n), \text{for $n\geq5$,}\\
    p_{eu}^{od}(n)&<p_{ou}^{ed}(n), \text{for $n\geq2$,}\\
    p_{od}^{ed}(n)&<p_{eu}^{od}(n), \text{for $n\geq8$.}
\end{align*}

In this section we mainly look at some inequalities between $p_{eu}^{ou}(n)$ and  $p_{ou}^{eu}(n)$. Unlike Ballantine and Welch \cite{ballantine2024combinatorial}, we do not put any additional conditions. We get the following two generating functions from Andrews \cite{andrews}.
\[
P_{eu}^{ou}(q) :=	\sum_{n\geq0} p_{eu}^{ou}(n)q^n =\frac{1}{(1-q)(q^2;q^2)_{\infty}},
\]
and
\[
P_{ou}^{eu}(q) :=	\sum_{n\geq0} p_{ou}^{eu}(n)q^n =\frac{1}{1-q}\bigg(\frac{1}{(q;q^2)_{\infty}}-\frac{1}{(q^2;q^2)_{\infty}}\bigg).
\]
Note that the set $P_{eu}^{ou}(n)$ includes the partitions with all parts even or odd. But $P_{ou}^{eu}(n)$ does not include the partitions with all parts even.

We now prove the following inequality between $p_{ou}^{eu}(n)$ and $p_{eu}^{ou}(n)$.
\begin{theorem}\label{Peu}
	For all $n>6$, we have
	\[p_{ou}^{eu}(n)> p_{eu}^{ou}(n).\]
\end{theorem}

\begin{proof}
We have
	\begin{align*}
		P_{ou}^{eu}(q) - P_{eu}^{ou}(q) &=\frac{1}{1-q}\bigg(\frac{1}{(q;q^2)_{\infty}}-\frac{2}{(q^2;q^2)_{\infty}}\bigg)=\frac{1}{(1-q)(q^2;q^2)_\infty}  \bigg(\frac{(q^2;q^2)_\infty^2}{(q;q)_\infty}-2\bigg)\\
		&=\frac{1}{(1-q)(q^2;q^2)_\infty}  \bigg(\sum_{n\geq0}q^{\frac{n^2+n}{2}}-2\bigg),
	\end{align*}

 \noindent where the last equality follows from \cite[p. 23]{gea1}. 
 
 We now note that the products on the RHS can be rewritten as
 \[
(1+q+q^2+q^3+\cdots)\prod_{i=1}^\infty(1+q^{2i}+q^{4i}+q^{6i}+\cdots)(-1+q+q^3+q^6+q^{10}+q^{15}+\cdots).
\]
Let $(1+q+q^2+q^3+\cdots)\prod\limits_{i=1}^\infty(1+q^{2i}+q^{4i}+q^{6i}+\cdots)=\sum\limits_{n\geq 0}a_nq^n$. Then we can prove that
\[
a_{2n}=a_{2n+1}, \quad \text{for all}~n\geq 0,
\]
and the series begins as
\[
1+q+2q^2+2q^3+4q^4+4q^5+7q^6+7q^7+\cdots,
\]
where the coefficients of $q^n$ are clearly monotonically non-decreasing. Multiplying this with $(-1+q+q^3+q^6+q^{10}+q^{15}+\cdots)$ now shows that indeed the coefficients of $q^{2n+1}$ in $P_{ou}^{eu}(q) - P_{eu}^{ou}(q)$ are nonnegative for $n\geq1$ (since each instance of $a_{2n+1}q^{2n+1}$ multiplied with $-1$ will be cancelled out by $a_{2n}q^{2n}$ multiplied with $q$).

Let $\prod\limits_{i=1}^\infty(1+q^{2i}+q^{4i}+q^{6i}+\cdots)=\sum\limits_{n\geq 0}b_{2n}q^{2n}$, where $b_{2n}$ is the number of partitions of $2n$ with all parts even. To prove that the coefficients of $q^{2n}$ in $P_{ou}^{eu}(q) - P_{eu}^{ou}(q)$ are nonnegative for $n\geq 4$, we have to prove that
\[a_{2n-1}+a_{2n-3}>a_{2n},\]
which means
\[a_{2n-2}+a_{2n-3}>a_{2n}.\]
It is easy to see that
\[a_{2n}=\sum_{i=0}^{n}b_{2i}, \quad \text{and} \quad a_{2n-3}=\sum_{i=0}^{n-2}b_{2i}.\]
This implies,
\[a_{2n-2}+a_{2n-3}-a_{2n} = \sum_{i=0}^{n-2}b_{2i} - b_{2n}.\]
So, to complete the proof, it is enough to show that
\begin{equation}\label{eq:in}
    \sum\limits_{i=0}^{n-2}b_{2i} - b_{2n}>0.
\end{equation}

This is not difficult to see combinatorially. We define the set $\tilde{P}(2n)$ to be the set of partitions of $2n$ into even parts. Let $\tilde{A}(2n)=\tilde{P}(2n)\setminus\{(2n),(\underbrace{2,2,\ldots,2}_{n})\}$. Then we define an injection $\varphi: \tilde{A}(2n) \rightarrow \bigcup \limits_{i=1}^{n-2} \tilde{P}(2i)$ by mapping any partition $\lambda$ in $\tilde{A}(2n)$ to a partition in $\tilde{P}(2i)$ for $1\leq i\leq n-2$ by removing the largest part of $\lambda$. And we map $(2n)$ to $(2n-4)$, and $(\underbrace{2,2,\ldots,2}_{n})$ to $(2n-6)$, which is possible for all $n\geq7$. This proves the inequality \eqref{eq:in} for $n\geq 7$. So, the coefficients of even powers of $q$ in  $P_{ou}^{eu}(q) - P_{eu}^{ou}(q)$ are positive for all $n\geq 14$. Verifying for the smaller even powers of $q$, we get the theorem.
\end{proof}

\begin{remark}
In fact, it is possible to prove combinatorially that, for all $n\geq 7$, we have
\[
b_{2n-4}+b_{2n-6}+b_{2n-8}+b_{2n-10}>b_{2n}.
\]
This will give an alternate justification of the previous proof without invoking the map $\varphi$.
\end{remark}

We also look at non-unitary versions of these types of partitions. Let us denote by $Q_{eu}^{ou}(n)$ and $Q_{ou}^{eu}(n)$ the set of non-unitary partitions which are in the sets $P_{eu}^{ou}(n)$ and $P_{ou}^{eu}(n)$ respectively. Let us denote the cardinalities of these two sets by $q_{eu}^{ou}(n)$ and $q_{ou}^{eu}(n)$, respectively. If $1$ is a part in any partition inside $P_{eu}^{ou}(n)$, then no even part is there in that partition. So, we get the following generating function.
\[Q_{eu}^{ou}(q):= \sum_{n\geq0} q_{eu}^{ou}(n)q^n =\frac{1}{(1-q)(q^2;q^2)_{\infty}} - \frac{q}{(q;q^2)_{\infty}}.\]
If $1$ is not a part in any partition inside $P_{ou}^{eu}(n)$, then the least odd part of that partition is greater than or equal to $3$. So, in any case, the partition can not contain $2$ as a part. Therefore, we get the following generating function (for details see Andrews \cite{andrews}).
\begin{align*}
	Q_{ou}^{eu}(q):= \sum_{n\geq0} q_{ou}^{eu}(n)q^n &= \sum_{n\geq0}\frac{q^{2n+3}}{(q^3;q^2)_{n+1}(q^{2n+4};q^2)_{\infty}}\\
	&= \frac{q}{(q^2;q^2)_{\infty}}\bigg(\sum_{n\geq0}\frac{q^{2n}(q^2;q^2)_{n}}{(q^3;q^2)_{n}}-1\bigg)\\
	&=\frac{1}{(q;q^2)_{\infty}}-\frac{q+1}{(q^2;q^2)_{\infty}}.
\end{align*}

We now have the following result.
\begin{theorem}\label{Qeu}
	For all $n>3$, we have
	\[q_{ou}^{eu}(n)< q_{eu}^{ou}(n).\]
\end{theorem}

\begin{proof} We have
\begin{align*}
	Q_{eu}^{ou}(q) - Q_{ou}^{eu}(q)
	&=\frac{2-q^2}{1-q} \cdot \frac{1}{(q^2;q^2)_{\infty}} - \frac{1+q}{(q;q^2)_{\infty}}\\
 	&=\frac{1}{(q^2;q^2)_\infty} \sum_{n\geq0} \bigg(\frac{ (2-q^2)(1-q^{n+1})}{1-q}-(1+q)\bigg)q^{\frac{n^2+n}{2}}\quad(\text{thanks to \eqref{eq:euler-2}}, \eqref{eq:euler-3}, \eqref{eq:heine-1}, \eqref{eq:heine-2})\\
 	&=\frac{1}{(q^2;q^2)_\infty} \bigg(\sum_{n\geq0}(1+q+q^2+\cdots+q^n)q^{\frac{n(n+1)}{2}} -\sum_{n\geq0}(1+q)q^{\frac{(n+1)(n+2)}{2}}\bigg)\\
 	&=\frac{1}{(q^2;q^2)_\infty} \bigg(1+\sum_{n\geq0}(1+q+q^2+\cdots+q^{n+1})q^{\frac{(n+1)(n+2)}{2}} -\sum_{n\geq0}(1+q)q^{\frac{(n+1)(n+2)}{2}}\bigg)\\
 	&=\frac{1}{(q^2;q^2)_\infty} \bigg(1+\sum_{n\geq1}(q^2+\cdots+q^{n+1})q^{\frac{(n+1)(n+2)}{2}}\bigg).
\end{align*}
Hence, the coefficients of $q^n$ in $Q_{eu}^{ou}(q) - Q_{ou}^{eu}(q)$ are positive for all $n>3$.
\end{proof}

\section{Concluding Remarks}\label{sec:conc}

There are several natural questions that arise from our study, including several avenues for further research. We list below a selection of such questions and comments.

\begin{enumerate}
\item Experiments suggest that the inequality in Theorem \ref{thm:mm} can be strengthened. We conjecture that, for all $n>9$, we have
\[
3q_o(n)<2q_e(n).
\]
In fact, it is easy to see that this is true for all even $n$, since we have
\begin{align*}
2P_{e}(q) - 3P_{o}(q)
& = \frac{1}{(q^3;q^2)_{\infty}}\sum_{n\geq1}\frac{q^{2n^2}}{(q^2;q^2)^2_{n}}(1-q^{n})^2(2+q^n),
\end{align*}
and when $n$ is even then $(1-q^{n})^2$ is canceled by a factor of $(q^2;q^2)_n^2$.
\item Chern \cite[Theorem 1.3]{zbMATH07481756} has recently proved for $m\geq 2$ and for integers $a$ and $b$ such that $1\leq a<b\leq m$, we have
\[
p_{a,b,m}(n)\geq p_{b,a,m}(n),
\]
thus generalizing the result of Kim and Kim \cite{zbMATH07394382}. Limited data suggest that this inequality is reversed if we consider $q_{j,k,m}(n)$ instead of $p_{j,k,m}(n)$. It would be interesting to obtain a unified proof of this observation.
\item Kim, Kim, and Lovejoy \cite{zbMATH07394382} and Kim and Kim \cite{zbMATH07394382} also study asymptotics of some of their parity biases. It would be interesting to study such asymptotics for our cases as well.
\item All the proofs in this paper are analytical. It would be interesting to obtain combinatorial proofs of some of these results.
\item Analytical proofs of the inequalities \eqref{eq:22} and \eqref{eq:33} would also be of interest to see if we can obtain more generalized results of a similar flavor.
\item Alanazi and Nyirenda \cite{AlanaziNyirenda} and Chern \cite{Chern2} study some more classes of partitions where the parts are separated by parity, following the work of Andrews \cite{andrews}. It would be interesting to see if inequalities of the type proved in Theorems \ref{Peu} and \ref{Qeu} can be proved for these cases as well as for other classes studied by Andrews \cite{andrews}.
\end{enumerate}

\section*{Acknowledgements}
The authors thank Professor Nayandeep Deka Baruah (Tezpur) for helpful comments and encouragement. This work was carried out when the third author was a PhD student at Tezpur University where he was supported by an institutional fellowship for doctoral research. The authors thank the anonymous reviewer for helpful comments, which improved the manuscript.

\section*{Declarations}
\textbf{Conflict of interest:} The authors declare that they have no conflict of interest.

\newcommand{\etalchar}[1]{$^{#1}$}

\end{document}